\documentclass{amsart}
\usepackage[utf8]{inputenc}
\usepackage{amsmath}
\usepackage{amsfonts}
\usepackage{amsthm}
\usepackage{amssymb}
\usepackage[svgnames]{xcolor}
\usepackage{wasysym } 
\usepackage{tikz}
\usepackage{comment}
\usetikzlibrary{decorations.pathreplacing,calligraphy}

\newcommand{\dave}[1]{{\color{red} ($\diamondsuit$ Dave: #1)}}
\newcommand{\noah}[1]{{\color{MediumSlateBlue} ($\clubsuit$ Noah: #1)}}

\newtheorem{theorem}{Theorem}[section]
\newtheorem{corollary}{Corollary}[theorem]
\newtheorem{lemma}[theorem]{Lemma}
\newtheorem{proposition}[theorem]{Proposition}
\newtheorem{question}[theorem]{Question}

\newcommand{\gon}{\operatorname{gon}}
\newcommand{\Supp}{\operatorname{Supp}}
\newcommand{\Deg}{\operatorname{deg}}
\newcommand{\rk}{\operatorname{rk}}
\newcommand{\banana}[2]{B^{*}_{#1,#2}}
\newcommand{\bananaStar}[3]{B^{*}_{#1,#2,#3}}
\newcommand{\reduced}[1]{#1\text{-reduced}}
\newcommand{\rook}[2]{K_{#1} \square K_{#2}}
\newcommand{\PP}{\mathbb{P}}
\newcommand{\cL}{\mathcal{L}}
\newcommand{\cO}{\mathcal{O}}

\title{On the Semigroup of Graph Gonality Sequences}
\author{Austin Fessler, David Jensen, Elizabeth Kelsey, Noah Owen }
\begin{document}

\bibliographystyle{alpha}

\maketitle

\begin{abstract}
The $r$th gonality of a graph is the smallest degree of a divisor on the graph with rank $r$.  The gonality sequence of a graph is a tropical analogue of the gonality sequence of an algebraic curve.  We show that the set of truncated gonality sequences of graphs forms a semigroup under addition.  Using this, we study which triples $(x,y,z)$ can be the first 3 terms of a graph gonality sequence.  We show that nearly every such triple with $z \geq \frac{3}{2}x+2$ is the first three terms of a graph gonality sequence, and also exhibit triples where the ratio $\frac{z}{x}$ is an arbitrary rational number between 1 and 3.  In the final section, we study algebraic curves whose $r$th and $(r+1)$st gonality differ by 1, and posit several questions about graphs with this property.
\end{abstract}

\section{Introduction}
The theory of divisors on graphs, developed by Baker and Norine in \cite{Baker08, BakerNorine09}, mirrors that of divisors on curves.  Two important invariants of a divisor $D$, on either a graph or a curve, are its degree $\Deg (D)$ and its rank $\rk(D)$.  For $r \geq 1$, the $r$th \emph{gonality} of a graph is the smallest degree of a divisor of rank $r$:
\[
\gon_r (G) := \min_{D \in \mathrm{Div} (G)} \{ \Deg(D) \mid \rk (D) \geq r \} .
\]
The \emph{gonality sequence} of a graph $G$ is the sequence:
\[
\gon_1 (G), \gon_2 (G) , \gon_3 (G), \ldots
\]
In \cite{ADMYY}, the authors ask which integer sequences are the gonality sequence of some graph. 

In this paper, we approach this problem by studying the first $r$ terms of the gonality sequence.  Let
\[ 
\mathcal{G}_r := \{ \vec{x} \in \mathbb{N}^r \mid \exists \text{ a graph } G \text{ with } \gon_k (G) = x_k \text{ for all } k \leq r \} .
\]
Our first main observation is that $\mathcal{G}_r$ is a \emph{semigroup} -- that is, it is closed under addition.  We say that an element $\vec{x} \in \mathcal{G}_r$ is \emph{reducible} if it can be written as the sum of two elements in $\mathcal{G}_r$.

\begin{theorem}
\label{Thm:Closed}
The set $\mathcal{G}_r$ is closed under addition.  Moreover, if $\vec{x} \in \mathcal{G}_r$ is reducible, then for all $g$ sufficiently large, there exists a graph $G$ of genus $g$ such that $\gon_k (G) = x_k$ for all $k \leq r$.
\end{theorem}

The set $\mathcal{G}_r$ is always contained in the cone:
\[
\mathcal{C}_r := \{ \vec{x} \in \mathbb{N}^r \mid x_i < x_{i+1} \text{ and } x_{i+j} \leq x_i + x_j \text{ for all } i,j \leq r \}
\]
(See Lemmas~\ref{Lem:Increasing} and~\ref{Lem:Subadditive}).
Using Theorem~\ref{Thm:Closed}, we give a short proof of \cite[Theorem~1.5]{ADMYY}.

\begin{theorem} \cite[Theorem~1.5]{ADMYY}
\label{Thm:G2}
We have 
\[
\mathcal{G}_2  = \mathcal{C}_2 = \{ (x,y) \in \mathbb{N}^3 \mid x+1 \leq y \leq 2x \} .
\]
Moreover, if $x+2 \leq y \leq 2x$, then for all sufficiently large $g$, there exists a graph $G$ of genus $g$ such that $\gon_1 (G) = x$ and $\gon_2 (G) = y$.
\end{theorem}

As noted in Section 4 of \cite{ADMYY}, Theorem~\ref{Thm:G2} demonstrates that there are graphs whose gonality sequence cannot be the gonality sequence of an algebraic curve.  For example, if $C$ is a curve whose 2nd gonality $\gon_2 (C) = p$ is prime, then $C$ maps generically 1-to-1 onto a plane curve of degree $p$.  It follows that the genus of $C$ is at most ${{p-1}\choose{2}}$.  On the other hand, if $p \geq 5$, then by Theorem~\ref{Thm:G2} there exists a graph $G$ of genus $g$ with $\gon_1 (G) = p-2$ and $\gon_2 (G) = p$ for all $g$ sufficiently large.  Since the genus of a graph is determined by its gonality sequence, we see that the gonality sequence of $G$ does not agree with that of any algebraic curve.

On the other hand, if $\gon_2 (G) = \gon_1 (G) + 1$, then $(\gon_1 (G), \gon_2 (G))$ is an irreducible element of $\mathcal{G}_2$.  We know of two infinite families of graphs such that the 2nd gonality is 1 greater than the 1st gonality -- the complete graph $K_{x+1}$ and the generalized banana graph $B^*_{x,x}$ from \cite{ADMYY}.  Interestingly, both graphs have genus ${{x}\choose{2}}$ and 3rd gonality $\gon_3 = 2x$.  This is exactly the genus and 3rd gonality of an algebraic curve $C$ satisfying $\gon_2 (C) = \gon_1 (C) + 1 = x+1$ (see Lemma~\ref{Lem:PlaneCurves}).  We ask whether this holds more generally.

\begin{question}
\label{Q:PlaneGraphs}
Let $G$ be a graph with the property that $\gon_2 (G) = \gon_1 (G) + 1$.  
\begin{enumerate}
\item Is the genus of $G$ necessarily $g = {{\gon_1 (G)}\choose{2}}$? 

\item For $r<g$, do we have
\[
\gon_r (G) = k \cdot \gon_2 (G) - h,
\]
where $k$ and $h$ are the uniquely determined integers with $1 \leq k \leq \gon_2(G) - 3$, $0 \leq h \leq k$, such that $r = \frac{k(k+3)}{2} - h$?

\item In particular, if $\gon_1 (G) \geq 2$, does it follow that $\gon_3 (G) = 2 \cdot \gon_1 (G)$?
\end{enumerate}
\end{question}

Much of this paper is dedicated to studying $\mathcal{G}_3$.  Unlike $\mathcal{G}_2$, we are unable to provide a complete description of $\mathcal{G}_3$.  However, we have the following partial result.

\begin{theorem}
\label{Thm:G3zBig}
Let $(x,y,z) \in \mathcal{C}_3$ with $z \geq 2x$.  Suppose further that:
\begin{itemize}
\item  if $y=x+1$, then $z = 2x$, and
\item  if $z=x+y$, then $y = 2x$.
\end{itemize}
Then $(x,y,z) \in \mathcal{G}_3$.
%\[
%\mathcal{G}_3 \supseteq \{ (x,y,z) \in \mathcal{C}_3 \mid 2x < z < x+y, x+1 < y \} \cup \{ (x,y,z) \in \mathcal{C}_3 \mid z=2x \} . 
%\]
%\[
%\mathcal{G}_3 \supseteq \{ (x,y,z) \in \mathcal{C}_3 \mid z \geq 2x \}  \smallsetminus \Big( \{ (x,x+1,z) \in \mathcal{C}_3 \mid z \neq 2x \} \cup \{ (x,y,x+y) \in \mathcal{C}_3 \mid y \neq 2x \}  \Big). 
%\]
\end{theorem}

We suspect that Theorem~\ref{Thm:G3zBig} classifies triples $(x,y,z) \in \mathcal{G}_3$ with $z \geq 2x$.  Indeed, by Lemmas~\ref{Lem:Increasing} and~\ref{Lem:Subadditive}, we have $\mathcal{G}_3 \subseteq \mathcal{C}_3$.  If $y=x+1$, then an affirmative answer to Question~\ref{Q:PlaneGraphs} would show that $z=2x$.  Similarly, if $z=x+y$, then an affirmative answer to \cite[Question~4.5]{ADMYY} would show that $y=2x$.  The goal of the rest of this paper is to study triples $(x,y,z) \in \mathcal{G}_3$ with $z < 2x$.  In Section~\ref{Sec:Sym}, we prove the following.

\begin{theorem}
\label{Thm:G3zLess}
Let $(x,y,z) \in \mathcal{C}_3$ with $x+2 \leq y \leq z-2$ and $z \geq \frac{3}{2}x+2$.
Then $(x,y,z) \in \mathcal{G}_3$.
\end{theorem}

Theorems~\ref{Thm:G3zBig} and~\ref{Thm:G3zLess} gives a possibly complete description of triples $(x,y,z) \in \mathcal{G}_3$ with $z \geq \frac{3}{2}x+2$.  However, there exist triples $(x,y,z) \in \mathcal{G}_3$ such that $z < \frac{3}{2}x + 2$.  Indeed,  we have the following.

\begin{lemma}
\label{Lem:CloseTo1}
Let $q$ be a rational number in the range $1 < q \leq 3$.  Then there exists $(x,y,z) \in \mathcal{G}_3$ such that $\frac{z}{x} = q$.
\end{lemma}

Unfortunately, it is difficult to write down a simple, closed-form expression for the semigroup generated by these triples.  It seems likely that the techniques of this paper could be used to study $\mathcal{G}_r$ for $r \geq 4$, or to produce analogues of Theorem~\ref{Thm:G3zLess} where the ratio $\frac{z}{x}$ is bounded below by a constant that is smaller than $\frac{3}{2}$.

The paper is organized as follows.  In Section~\ref{Sec:Prelims}, we present background on the divisor theory of graphs.  In Section~\ref{Sec:DP} we introduce graphs with known 1st, 2nd, and 3rd gonalities.  In Section~\ref{Sec:FirstProofs}, we prove all of the main results except for Theorem~\ref{Thm:G3zBig}, which is proved in Sections~\ref{Sec:Bananas} and~\ref{Sec:Sym}.  Finally, in Section~\ref{Sec:Curves}, we study the gonality sequences of certain algebraic curves, and ask several questions about graphs with the same gonality sequences.

\bigskip

\noindent{\textbf{Acknowledgments.}}  This research was conducted as a project with the University of Kentucky Math Lab, supported by NSF DMS-2054135.

\section{Preliminaries}
\label{Sec:Prelims}

In this section we will introduce the notion of gonality on graphs, along with important terms and concepts.  Throughout, we allow graphs to have parallel edges, but no loops.

A \emph{divisor} on a graph $G$ is a formal $\mathbb{Z}$-linear combination of the vertices in $G.$ A divisor $D$ can be expressed as 
\[ 
D = \sum_{v \in V(G)} D(v) \cdot v, 
\]
where each $D(v)$ is an integer.
The \emph{degree} of a divisor $D$, denoted $\deg(D)$, is the sum of the coefficients of $D$.
The \emph{support} of a divisor, denoted $\Supp(D)$, is defined as 
\[\Supp(D) = \{v \in V(G) | D(v) > 0\}\]

It is standard to think about divisors on graphs in terms of chip configurations. In a chip configuration, the coefficient of a vertex $v$ is reinterpreted as the number of chips sitting on $v.$ So, in a divisor $D,$ $v$ has $D(v)$ chips sitting on it. A vertex with a negative number of chips is said to be ``in debt.''  A divisor is \emph{effective} if, for every $w \in V(G)$, we have $D(w) \geq 0$. In other words, a divisor is effective if there are no vertices in debt.  A divisor is \emph{effective away from $v$} if, for every $w \in V(G) \setminus \{v\}$, we have $D(w) \geq 0$.

From this interpretation we can define a chip-firing move.  Firing a vertex $v$ causes $v$ to redistribute some of its chips by passing one chip across each of the edges incident to it.  We say that two divisors $D$ and $D^\prime$ are \emph{equivalent} if $D^\prime$ can be obtained from $D$ via a sequence of chip firing moves. The \emph{rank} of a divisor $D,$ denoted $\rk(D)$, is the largest integer $r$ such that $D - E$ is equivalent to an effective divisor for every effective divisor $E$ of degree $r$.  The $r$th \emph{gonality} of a graph is the minimum degree over all divisors of rank $r$.

Gonality is often framed as a chip firing game. Given a starting divisor we allow the ``opponent'' of the game to remove $r$ chips from anywhere on the graph.  A divisor has rank $r$ if, for every choice of chips by the opponent, there is a sequence of chip firing moves that eliminates all debt on the graph.

We recall some basic facts about the $r$th gonality from \cite{ADMYY}.

\begin{lemma}\cite[Lemma~3.1]{ADMYY}
\label{Lem:Increasing}
Let $G$ be a graph.  For all $r$, we have $\gon_r (G) < \gon_{r+1} (G)$.
\end{lemma}

\begin{lemma}\cite[Lemma~3.2]{ADMYY}
\label{Lem:Subadditive}
Let $G$ be a graph.  For all $r$ and $s$, we have $\gon_{r+s} (G) \leq \gon_r (G) + \gon_s (G)$.
\end{lemma}

For a graph $G$ and a vertex $v \in V(G)$ we say that a divisor $D$ is $v$-\emph{reduced} if the following conditions are satisfied:
\begin{enumerate}
    \item $D$ is effective away from $v$, and
    \item for any subset $A \subseteq V(G) \smallsetminus \{v\}$, the divisor $D^\prime$ obtained by firing the all vertices in $A$ is not effective.
\end{enumerate}

Given a divisor $D$ and a vertex $v$, there exists a unique divisor equivalent to $D$ that is $v$-reduced.  Dhar's Burning Algorithm is a procedure that produces this unique representative.

Given a divisor $D$ and a vertex $v$, we produce the unique $v$-reduced divisor equivalent to $D$ by performing Dhar's burning algorithm as follows:

\begin{enumerate}
    \item Replace $D$ with a divisor that is effective away from $v$.
    \item Start a fire by burning vertex $v$.
    \item Burn every edge that is incident to a burnt vertex.
    \item Let $U$ be the set of unburnt vertices.  For each $w \in U$ we burn $w$ if the number of burnt edges incident to $w$ is strictly greater than $D(w)$.  If no new vertices in $U$ were burnt proceed to step (5).  Otherwise return to step (3).
    \item Let $U$ be the set of unburnt vertices.  If $U$ is empty, then $D$ is $v$-reduced and the algorithm terminates.  Otherwise, replace $D$ with the equivalent divisor $D^\prime $ obtained by firing all vertices in $U$ and return to step (2).
\end{enumerate}

Note that a divisor is $v$-reduced if and only if starting a fire at $v$ results in the entire graph being burnt. This makes Dhar's burning algorithm useful for determining if a divisor has positive rank. For any$v$-reduced divisor $D$, if $D(v) < 0,$ then $D$ does not have positive rank.

\section{Dramatis Personae}
\label{Sec:DP}

This section surveys graphs for which the first few terms of the gonality sequence are known.
The first of these graphs is the complete graph $K_n$, which has genus $g = {{n-1}\choose{2}}$.

\begin{lemma}\cite[Theorem~1]{CoolsPanizzut}
\label{Lem:Complete}
For $r<g$, the $r$th gonality of the complete graph $K_n$ is $\gon_r (K_n) = kn - h$, where $k$ and $h$ are the uniquely determined integers with $1 \leq k \leq n-3$, $0 \leq h \leq k$, such that $r = \frac{k(k+3)}{2} - h$.  In particular, if $n \geq 3$, then
\begin{align*}
\gon_1 (K_n) &= n-1 \\
\gon_2 (K_n) &= n \\
\gon_3 (K_n) &= 2n-2.
\end{align*}
\end{lemma}

Next, we have the complete bipartite graph $K_{m,n}$, which has genus $g = (m-1)(n-1)$.  Let
\[
I_r = \{ (a,b,h) \in \mathbb{N}^3 \mid a \leq m-1, b \leq n-1, \text{ and } r=(a+1)(b+1)-1-h \} ,
\]
and let
\[
\delta_r (m,n) = \min \{ an+bm-h \mid (a,b,h) \in I_r \} .
\]

\begin{lemma}\cite[Theorem~4]{CDJP}
\label{Lem:CompleteBipartite}
For $r<g$, the $r$th gonality of the complete bipartite graph $K_{m,n}$ is $\gon_r (K_{m,n}) = \delta_r (m,n)$.  In particular, if $2 \leq m \leq n$, then
\begin{align*}
\gon_1 (K_{m,n}) &= m \\
\gon_2 (K_{m,n}) &= \min \{ 2m,m+n-1 \} \\
\gon_3 (K_{m,n}) & = \min \{ 3m,m+n \}.
\end{align*}
\end{lemma}

The \emph{banana graph} $B_n$ is the graph consisting of $2$ vertices with $n$ edges connecting them. A \emph{generalized banana graph} is a graph with vertex set $\{v_1, \dots, v_n\}$ such that for each $1 \leq i < n,$ there is at least $1$ edge between $v_i$ and $v_{i+1}$ and no edges elsewhere.

In \cite{ADMYY}, the authors study the gonalty sequences of different families of generalized banana graphs.  The generalized banana graph $B_{n,e}$ is the graph with vertex set $\{v_1, \dots v_n\}$ and where there are $e$ edges between $v_i$ and $v_{i+1}$ for $1 \leq i \leq n-1.$  The generalized banana graph $\banana{a}{b}$ is the graph with vertex set $\{ v_1, \dots, v_a \}$ and with $b-a+i+1$ edges between $v_i$ and $v_{i+1}$ for $1 \leq i \leq a-1.$ The generalized banana graphs $B_{4,3}$ and $\banana{4}{5}$ are depicted in Figure \ref{fig:BananaGraphs1}.

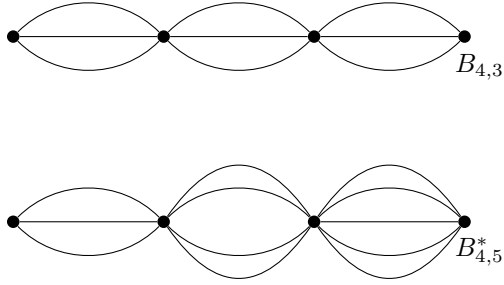
\begin{figure}
    \centering
    \begin{tikzpicture}[scale = 2]
\begin{scope}
\tikzstyle{every node}=[circle,fill,scale=0.5pt]
\node (0) at (0,0) {};
\node (1) at (1,0) {};
\node (2) at (2,0) {};
\node (3) at (3,0) {};

\foreach \i\j in {0/1,1/2,2/3}
{
\draw (\i) -- (\j);
\draw (\i) to [out=45,in=135,looseness=1] (\j);
\draw (\i) to [out=315,in=225,looseness=1] (\j);
}
\end{scope}

\node at (3.1,-0.2) {$B_{4,3}$};

\end{tikzpicture}

\vspace{0.75cm}

\begin{tikzpicture}[scale=2]
\begin{scope}
\tikzstyle{every node}=[circle,fill,scale=0.5pt]
\node (0) at (0,0) {};
\node (1) at (1,0) {};
\node (2) at (2,0) {};
\node (3) at (3,0) {};
\end{scope}

\draw (0) -- (1);
\draw (0) to [out=45,in=135,looseness=1] (1);
\draw (0) to [out=315,in=225,looseness=1] (1);

\draw (1) to [out=45,in=135,looseness=1] (2);
\draw (1) to [out=315,in=225,looseness=1] (2);

\draw (1) to [out=55,in=125,looseness=1.5] (2);
\draw (1) to [out=305,in=235,looseness=1.5] (2);

\draw (2) -- (3);
\draw (2) to
[out=45,in=135,looseness=1] (3);
\draw (2) to [out=315,in=225,looseness=1] (3);
\draw (2) to [out=55,in=125,looseness=1.5] (3);
\draw (2) to [out=305,in=235,looseness=1.5] (3);

\node at (3.1, -0.2) {$\banana{4}{5}$};

\end{tikzpicture}

\caption{The generalized banana graphs $B_{4,3}$ and $\banana{4}{5}.$}
    \label{fig:BananaGraphs1}
\end{figure}

\begin{lemma}\cite[Lemmas~5.2-5.4]{ADMYY}
\label{Lem:Bne}
We have
\begin{align*}
\gon_1 (B_{n,e}) &= \min \{ n,e \} \\
\gon_2 (B_{n,e}) &= \min \{ 2n, 2e, n+e-1 \} .
\end{align*}
\end{lemma}

\begin{lemma}\cite[Lemmas~5.5 and~5.6]{ADMYY}
\label{Lem:BStar}
If $2 \leq a \leq b \leq 2a-1$, we have
\begin{align*}
\gon_1 (\banana{a}{b}) &= a \\
\gon_2 (\banana{a}{b}) &= b+1 .
\end{align*}
\end{lemma}

The \emph{2-dimensional n by m rook graph} is the Cartesian product of the complete graphs $K_n$ and $K_m$.  The vertices can be thought of as the squares of an $n \times m$ chessboard, in which two vertices are adjacent if a rook can move from one to the other.  By convention, we assume throughout that $m \geq n$.  In \cite{Speeter}, Speeter computes the first 3 gonalities of these rook graphs.

\begin{lemma}\cite{Speeter}
\label{Lem:Rook}
If $2 \leq n \leq m$, then
\begin{align*}
\gon_1(K_n \square K_m) &= (n-1)m \\
\gon_2(K_n\square K_m) &= nm -1 \\
\gon_3(K_n\square K_m) &= nm.
\end{align*}
\end{lemma}

\section{Proofs of Theorems~\ref{Thm:Closed}-\ref{Thm:G3zBig}}
\label{Sec:FirstProofs}

In this section, we prove many of the main theorems.  The central construction is the following.  Given two graphs $G_1$ and $G_2$, and vertices $v_1 \in V(G_1)$, $v_2 \in (G_2)$, we ``put them together'' by connecting $v_1$ to $v_2$ with a number of parallel edges, as in Figure~\ref{fig:glue}.

\begin{figure}[h]
    \centering
    \begin{tikzpicture}

\draw (0,0) circle (2);
\draw (5,0) circle (2);
\draw (0,0) node {$G_1$};
\draw (5,0) node {$G_2$};
\draw (1.5,-0.3) node {\footnotesize$v_1$};
\draw (3.5,-0.3) node {\footnotesize$v_2$};
\draw (2.5,-1) node {\footnotesize$\ell$};

\draw (1.5,0) -- (3.5,0);
\draw (1.5,0) to [out=45,in=135,looseness=1] (3.5,0);
\draw (1.5,0) to [out=315,in=225,looseness=1] (3.5,0);
\draw (1.5,0) to [out=55,in=125,looseness=1.5] (3.5,0);
\draw (1.5,0) to [out=305,in=235,looseness=1.5] (3.5,0);

\end{tikzpicture}
\caption{Gluing two graphs together}
    \label{fig:glue}
\end{figure}
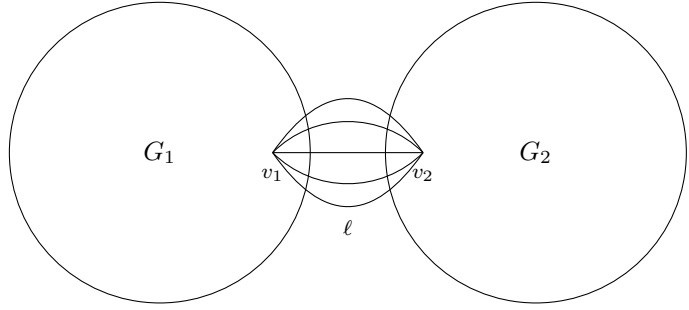

\begin{lemma}
\label{Lem:GlueReduced}
Let $G_1$ and $G_2$ be graphs, let $v_1 \in V(G_1)$, $v_2 \in V(G_2)$, and let $G$ be the graph obtained by connecting $v_1$ to $v_2$ with $\ell$ edges, as in Figure~\ref{fig:glue}.  If $D$ is a $v_1$-reduced divisor on $G$, then $\rk(D \vert_{G_1}) \geq \rk(D)$.
\end{lemma}

\begin{proof}
Let $E$ be an effective divisor of degree $\rk (D)$ on $G_1$.  By definition, $D-E$ is equivalent to an effective divisor.  There exists a sequence of subsets:
\[
 U_1 \subseteq U_2 \subseteq \cdots \subseteq U_k \subset V(G)
\]
and a sequence of effective divisors $D_0 , \ldots , D_k$ such that:
\begin{enumerate}
\item  $D_0 = D$,
\item  $D_k - E$ is effective, and
\item  $D_i$ is obtained from $D_{i-1}$ by firing $U_i$.
\end{enumerate}
Since $D_0$ is $v_1$-reduced and $D_1$ is effective, we must have $v_1 \in U_1$.  Thus, $v_1 \in U_i$ for all $i$.

Now, consider the sequence of subsets $W_i = U_i \cup V(G_2)$.  Let $D'_0 = D$ and let $D'_i$ be the divisor obtained from $D'_{i-1}$ by firing $W_i$.  Since $V(G_2) \cup \{ v_1 \} \subseteq W_i$ for all $i$, we have $D'_i (v) = D(v)$ for all $v \in V(G_2)$.  Since $D$ is $v_1$-reduced, it follows that $D'_k (v) \geq 0$ for all $v \in V(G_2)$.  Note also that $D'_i (v) = D_i (v)$ for all $v \in V(G_1) \smallsetminus \{ v_1 \}$, and $D'_i (v_1) \geq D_i (v_1)$.  It follows that $D'_k - E$ is effective.

Finally, note that firing $W_i$ passes no chips from $G_1$ to $G_2$ or from $G_2$ to $G_1$.  Thus, $D \vert_{G_1}$ is equivalent to $D'_k \vert_{G_1}$ on $G_1$.  In this way, $D \vert_{G_1} - E$ is equivalent on $G_1$ to an effective divisor.  Since $E$ was arbitrary, we see that $\rk(D \vert_{G_1}) \geq \rk(D)$. 
\end{proof}

If the number of parallel edges between $v_1$ and $v_2$ is large enough, then the $r$th gonality of the graph $G$ is the sum of the $r$th gonalities of the graphs $G_1$ and $G_2$.

\begin{proposition}
\label{Prop:ClosedUnderSum}
Let $G_1$ and $G_2$ be graphs, let $v_1 \in V(G_1)$, $v_2 \in V(G_2)$, and let $G$ be the graph obtained by connecting $v_1$ to $v_2$ with $\ell$ edges.  If $\ell \geq \gon_r (G_1) + \gon_r (G_2)$, then
\[
\gon_k (G) = \gon_k (G_1) + \gon_k (G_2) \mbox{ for all } k \leq r .
\]
\end{proposition}

\begin{proof}
Let $k \leq r$, let $D_1$ be a divisor of rank $k$ and degree $\gon_k (G_1)$ on $G_1$, and let $D_2$ be a divisor of rank $k$ and degree $\gon_k (G_2)$ on $G_2$.  Then the divisor $D_1 + D_2$ has rank at least $k$ on $G$, so $\gon_k (G) \leq \gon_k (G_1) + \gon_k (G_2)$.

For the reverse inequality, let $D$ be a divisor of rank at least $k$ on $G$.  We must show that $\Deg (D) \geq \gon_k (G_1) + \gon_k (G_2)$.  If $\Deg (D \vert_{G_i}) \geq \gon_k (G_i)$ for $i=1,2$, then $\Deg (D) \geq \gon_k (G_1) + \gon_k (G_2)$.  On the other hand, suppose without loss of generality that $\Deg (D \vert_{G_1}) < \gon_k (G_1)$.  Then $D \vert_{G_1}$ has rank less than $k$, so we must be able to pass chips from $G_2$ to $G_1$.  Since there are $\ell$ edges between $G_1$ and $G_2$, it follows that $\Deg(D) \geq \ell \geq \gon_k (G_1) + \gon_k (G_2)$.
\end{proof}

Theorem~\ref{Thm:Closed} is a direct corollary.

\begin{proof}[Proof of Theorem~\ref{Thm:Closed}]
Let $\vec{x}, \vec{y} \in \mathcal{G}_r$.  By definition, there exist graphs $G_1$ and $G_2$ such that $\gon_k (G_1) = x_k$ and $\gon_k (G_2) = y_k$ for all $k \leq r$.  By Proposition~\ref{Prop:ClosedUnderSum}, there exists a graph $G$ with $\gon_k (G) = x_k + y_k$ for all $k \leq r$.  Moreover, if $G_1$ has genus $g_1$ and $G_2$ has genus $g_2$, then by Proposition~\ref{Prop:ClosedUnderSum}, for any $\ell \geq \gon_r (G_1) + \gon_r (G_2)$, there exists such a graph $G$ of genus $g = g_1 + g_2 + \ell$.
\end{proof}

Using the fact that $\mathcal{G}_2$ is closed under addition, we provide a short proof of Theorem~\ref{Thm:G2}.

\begin{proof}[Proof of Theorem~\ref{Thm:G2}]
Let $G$ be a graph.  By Lemma~\ref{Lem:Increasing}, we have $\gon_2 (G) \geq \gon_1 (G) + 1$.  By Lemma~\ref{Lem:Subadditive}, $\gon_2(G) \leq 2 \cdot \gon_1 (G)$.  In other words, $\mathcal{G}_2 \subseteq \mathcal{C}_2$.

We now show the reverse containment.  In other words, we show that if $x+1 \leq y \leq 2x$, then there exists a graph $G$ with $\gon_1 (G) = x$ and $\gon_2 (G) = y$.  We proceed by induction on $y-x$.  For the base case, when $y=x+1$, by Lemma~\ref{Lem:Complete}, the complete graph on $x+1$ vertices $K_{x+1}$ satisfies
\[
x = \gon_1 (K_{x+1}) = \gon_2 (K_{x+1}) - 1 .
\]
For the inductive step, if $y \geq x+2$, then since $(y-2)-(x-1) = y-x-1$, by induction $(x-1,y-2) \in \mathcal{G}_2$.  If $T$ is a tree, then $\gon_r (T) = r$ for all $r$, so $(1,2) \in \mathcal{G}_2$.  By Theorem~\ref{Thm:Closed}, therefore, $(x,y)$ is a reducible element of $\mathcal{G}_2$, and the result follows.
\end{proof}

A similar strategy allows us to construct triples $(x,y,z) \in \mathcal{G}_3$ where $z$ is large relative to $x$.

\begin{proof}[Proof of Theorem~\ref{Thm:G3zBig}]
Let $(x,y,z) \in \mathbb{N}^3$ satisfy $x<y<z$, $y \leq 2x$, $z \leq x+y$.  We first show that, if $z=2x$, then $(x,y,z) \in \mathcal{G}_3$.  If $y=x+1$, then by Lemma~\ref{Lem:Complete}, the first 3 terms of the gonality sequence of the complete graph $K_y$ are $(x,x+1,2x)$, so $(x,x+1,2x) \in \mathcal{G}_3$.  Similarly, if $y=2x-1$, then by Lemma~\ref{Lem:CompleteBipartite}, the first 3 terms of the gonality sequence of the complete bipartite graph $K_{x,x}$ are $(x,2x-1,2x)$, so $(x,2x-1,2x) \in \mathcal{G}_3$.  Otherwise, if $x+2 \leq y \leq 2x-2$, then the first 3 terms of the gonality sequence of the complete graph $K_{2x-y+1}$ are $(2x-y, 2x-y+1, 4x-2y)$ and the first 3 terms of the gonality sequence of the complete bipartitie graph $K_{y-x,y-x}$ are $(y-x, 2y-2x-1, 2y-2x)$.  By Theorem~\ref{Thm:Closed}, therefore,  we have
\[
(2x-y, 2x-y+1, 4x-2y) + (y-x, 2y-2x-1, 2y-2x) = (x, y, 2x) \in \mathcal{G}_3 .
\]

We now consider cases where $2x < z \leq x+y-1$.  If $y=2x$, then the first 3 terms of the gonality sequence of the complete bipartitie graph $K_{x,z-x}$ are $(x,2x,z)$, so $(x,2x,z) \in \mathcal{G}_3$.  If $y=x+2$, then by assumption, $z = 2x+1$.  As above, the first 3 terms of the gonality sequence of the complete graph $K_x$ are $(x-1, x, 2x-2)$ and the first 3 terms of the gonality sequence of a tree are $(1, 2, 3)$.  By Theorem~\ref{Thm:Closed}, therefore,  we have
\[
(x-1, x, 2x-2) + (1, 2, 3) = (x, x+2, 2x+1) \in \mathcal{G}_3 .
\]

Finally, we show that, if $x+3 \leq y \leq 2x-1$, then $(x,y,z) \in \mathcal{G}_3$.  Similar to the above, the first 3 terms of the gonality sequence of the complete graph $K_{2x-y+2}$ are $(2x-y+1, 2x-y+2, 4x-2y+2)$.  Since $z > 2x$, we have $z+y-3x-1 > y-x-1$, and since $z \leq x+y-1$, we have $z+y-3x-1 \leq 2(y-x-1)$.  It follows that the first 3 terms of the gonality sequence of the complete bipartitie graph $K_{y-x-1,z+y-3x-1}$ are $(y-x-1, 2y-2x-2, 2y-4x-2+z)$.  By Theorem~\ref{Thm:Closed}, therefore,  we have
\[
(2x-y+1, 2x-y+2, 4x-2y+2) + (y-x-1, 2y-2x-2, 2y-4x-2+z) = (x,y,z) \in \mathcal{G}_3 .
\]
\end{proof}

\begin{proof}[Proof of Lemma~\ref{Lem:CloseTo1}]
If $q \geq 2$, the conlcusion follows from Theorem~\ref{Thm:G3zBig}.  If $1 < q < 2$, then there exists an integer $n \geq 2$ such that $\frac{n+1}{n} \leq q < \frac{n}{n-1}$.  If $q = \frac{n+1}{n}$, then by Lemma~\ref{Lem:Rook}, for all $m \geq n+1$, we have $(nm, (n+1)m-1, (n+1)m) \in \mathcal{G}_3$, and the conclusion follows.

If $q > \frac{n+1}{n}$, let $\epsilon_1 = q - \frac{n+1}{n}$, let $\epsilon_2 = \frac{n-(n-1)q}{n+1}$, and let $\epsilon = \min \{ \epsilon_1 , \epsilon_2 \}$.  By assumption, $\epsilon > 0$.  We can therefore write $q = \frac{z}{x}$, where $x \geq \frac{1}{\epsilon}$.  Finally, let $m = nz - (n+1)x$ and let $m' = nx - (n-1)z$.  By construction, $m \geq n$ and $m' \geq n+1$.  By Lemma~\ref{Lem:Rook} and Theorem~\ref{Thm:Closed}, we have
\[
((n-1)m, nm-1, nm) + (nm', (n+1)m'-1, (n+1)m')  = (x, z-2, z) \in \mathcal{G}_3 . 
\]
\end{proof}

\section{Third Gonality of the Graphs $B^*_{a,b}$}
\label{Sec:Bananas}

The 1st and 2nd gonalities of the graph $B^*_{a,b}$ are computed in \cite{ADMYY}.  In this section, we compute the 3rd gonalities of these graphs.  We first consider divisors on $B^*_{a,b}$ of rank 3 with a large number of chips on $v_a$.

\begin{lemma}
\label{Lem:Banana3divisor}
Let $D$ be a divisor of at least rank $3$ on the graph $B^{*}_{a,b}$.  If $D(v_{a}) \geq b + 1$, then $\Deg(D) \geq a + b$.
\end{lemma}

\begin{proof}
For the base case consider the banana graph $B^{*}_{2,b}$. We will proceed by cases.
\begin{enumerate}
\item If $D(v_2) = b + 1 $, consider the divisor $D - 2 \cdot (v_2) - (v_1)$.  This divisor has $b-1$ chips on $v_2$, so we must have $D(v_1) \geq 1$.  Hence, $\Deg(D) \geq a + b.$
\item If $D(v_2) \geq b + 2$, then $\Deg(D) \geq b + 2 = a + b$.
\end{enumerate} 

Now for the induction step assume that the theorem holds for the banana graph $B^{*}_{a-1,b-1}$.  If $D(v_a) \geq a+b$, we are done.  If $D(v_a) = b+1$, then by Lemma~\ref{Lem:GlueReduced}, we see that the restriction of $D$ to $B^*_{a-1,b-1}$ must have rank at least 1.  By Lemma~\ref{Lem:BStar}, it follows that $\Deg (D \vert_{B^*_{a-1,b-1}}) \geq a$, hence $\Deg (D) \geq a+b+1$.  Finally, if $b+2 \leq D(v_a) < a+b$, consider the equivalent divisor obtained by firing $v_a$.  This divisor has at least 2 chips on $v_a$.  Note that $v_a$ can only be fired once since $a < b$.  By Lemma~\ref{Lem:GlueReduced}, the restriction of this divisor to the subgraph $B^{*}_{a-1,b-1}$ must have rank at least 3 and there are at least $b$ chips on $v_{a-1}$, so by the inductive hypothesis there are least $a + b - 2$ chips on this subgraph.  So $\Deg(D) \geq a + b$.
\end{proof}

\begin{corollary}
\label{Cor:Banana3divisor}
If $2a < b$, there is no divisor of rank at least 3 on the graph $B^{*}_{a.b}$ with $D(v_{a}) \geq b + 1$ and $\Deg(D) \leq 3a$.
\end{corollary}
\begin{proof}
By Lemma \ref{Lem:Banana3divisor}, a divisor $D$ of rank at least $3$ with $D(v_{a}) \geq b +1 $ must have $\Deg(D) \geq a + b > 3a$.
\end{proof}

We also consider divisors on $B^*_{a,b}$ of rank 3 with a small number of chips on $v_a$.

\begin{lemma}
\label{Lem:banana3divisor2}
Let $D$ be a divisor on $\banana{a}{b}$ of rank at least $3$.  If $D(v_a) \leq 2$, then $\Deg(D) \geq a + b$.
\end{lemma}

\begin{proof}
Assume without loss of generality that $D$ is $\reduced{v_{a-1}}$.  We proceed by cases.
\begin{enumerate}
\item  If $D(v_a) = 0$, then consider the divisor $D - v_a$. The resulting divisor has a debt on $v_a$, so we must have $D(v_{a-1}) \geq b$.  After moving $b$ of these chips to $v_a$ and subtracting one, by Lemma~\ref{Lem:GlueReduced}, the remaining divisor must have rank at least 2 on $\banana{a-1,}{b-1}$.  Hence, by Lemma~\ref{Lem:BStar}, there must be at least $b$ more chips on the rest of the graph.  So $\Deg(D) \geq 2b \geq a+b$.
\item  If $D(v_a) = 1$, then consider the divisor $D - 2 \cdot (v_a)$. The resulting divisor has a debt on $v_a$, so $D(v_{a-1}) \geq b$.  After moving $b$ of these chips to $v_a$ and subtracting one, by Lemma~\ref{Lem:GlueReduced} the remaining divisor must have rank at least 1 on $\banana{a-1}{b-1}$.  Hence, by Lemma~\ref{Lem:BStar}, there must be at least $a-1$ chips on the rest of the graph.  Therefore, $\Deg(D) \geq a + b$. 
\item  If $D(v_a) = 2$, then consider the divisor $D - 3 \cdot v_a$.  There is now a debt of $-1$ on $v_{a}$, so again there must be at least $b$ chips on $v_{a-1}$.  By Lemma~\ref{Lem:Banana3divisor}, there must be at least $a + b - 2$ chips on the subgraph induced by the vertices $\{v_{1}, \dots v_{a-1}\}$.  Thus, $\Deg(D) \geq (a + b -2 ) + 2 = a + b$.
\end{enumerate}
\end{proof}

Our computation of the 3rd gonality of $B^*_{a,b}$ will proceed by induction on $a$.  The following lemma establishes the base case, when $a=2$.

\begin{lemma}
\label{Lem:Banana3gonBaseCase}
If $b \geq 4,$ then $\gon_3 (B^{*}_{2,b}) = 6$.
\end{lemma}

\begin{proof}
As with any graph $\gon_3 (\banana{2}{b}) \leq 3|V(\banana{2}{b})| = 6$.  Assume there is a divisor $D$ with $\Deg(D) < 6$.  By symmetry, we may assume that $D(v_2) \leq 2$.  We proceed by cases. 
\begin{enumerate}
\item  If $D(v_2) = 0$, then $D(v_1) \leq 5$.  Consider the divisor $D - 2 \cdot (v_1) - (v_2)$. This divisor has a debt of $-1$ on $v_2$ but at most $3$ chips on $v_1$, so $D$ cannot be rank at least 3.
\item  If $D(v_2) = 1$, then $D(v_1) \leq 4$.  Consider the divisor $D - (v_1) - 2 \cdot (v_2)$. This divisor has a debt of $-1$ on $v_2$ but at most 3 chips on $v_1$, so $D$ cannot be rank at least 3.
\item  If $D(v_2) = 2$, then $D(v_1) \leq 3$.  Consider the divisor $D - 3 \cdot (v_2)$. This divisor has a debt of $-1$ on $v_2$ but at most $3$ chips on $v_1$, so $D$ cannot be rank at least 3.
\end{enumerate}
We conclude that $\gon_3 (\banana{2}{b}) = 6$.
\end{proof}

\begin{lemma}
\label{Lem:banana3gon}
If $b \geq 2a$, then $\gon_{3}(B^{*}_{a,b}) = 3a$.
\end{lemma}

\begin{proof}
As with any graph, $\gon_{3}(B^{*}_{a,b}) \leq 3|V(B^*_{a,b,n})| = 3a$.  Now, let $D$ be a divisor $D$ of rank $3$, and assume without loss of generality that $D$ is $v_{a-1}$-reduced.  If $\Deg(D) \geq a+b$, we are done.  If not, by Lemma \ref{Lem:banana3divisor2}, $D(v_a) \geq 3$.  We will proceed by induction on $a$.  The base case is Lemma \ref{Lem:Banana3gonBaseCase}.  Now assume that $\gon_3 (B^{*}_{a-1,b-1}) = 3(a-1)$.  By Lemma~\ref{Lem:GlueReduced}, the restriction of $D$ to $B^*_{a-1,b-1}$ must have rank at least $3$, so there are at least $3(a-1)$ chips on that subgraph.  It follows that $\Deg(D) \geq 3a$.
\end{proof}

\begin{theorem}
\label{Thm:Banana3gon}
If $a \leq b \leq 2a - 1$ then $\gon_{3}(B^{*}_{a,b}) = a + b$.
\end{theorem}
\begin{proof}
First note that the divisor 
\[ 
(b+1) \cdot (v_a) + \sum\limits_{i=1}^{a-1} v_i
\]
has rank at least 3 and degree $a+b$, so $\gon_3 (B^{*}_{a,b}) \leq a + b$.  Now, let $D$ be a divisor of rank at least 3.  If $\Deg (D) \geq a+b$, we are done.  If not, by Lemma~\ref{Lem:banana3divisor2}, we have $D(v_a) \geq 3$.  We proceed by induction on $a$.  For the base cases, we have $\gon_3 (\banana{2}{2}) = 4$ and $\gon_3 (\banana{2}{3}) = 5$, by the Riemann-Roch Theorem for graphs \cite[Theorem~1.12]{BakerNorine09}, and $\gon_3 (\banana{a}{2a}) = 3a$ by Lemma \ref{Lem:banana3gon}.  For the induction step assume that $\gon_3 (B^{*}_{a-1,b-1}) = a + b - 2$.  Since $D(v_a) \geq 3$, therefore, we have $\Deg(D) > a + b$.
\end{proof}

\section{Symmetric Generalized Banana Graphs}
\label{Sec:Sym}

Our goal now is to find a large family of graphs $G$ such that $\gon_3 (G) < 2 \cdot \gon_1 (G)$.  In this section, we consider a family of generalized banana graphs.  Let $B^{0,0}_{a,b,k}$ be the graph obtained from 2 copies of $B^*_{a,b}$ by connecting the two vertices of highest degree with $k$ edges, as in Figure~\ref{fig:SymBanana}.  More precisely, let $v_1 , \ldots , v_a$ be the vertices in $L = B^*_{a,b}$ and let $w_1 , \ldots , w_a$ be the vertices in $R = B^*_{a,b}$.  Then $B^{0,0}_{a,b,k}$ is the graph obtained by connecting $v_a$ to $w_a$ with $k$ edges.

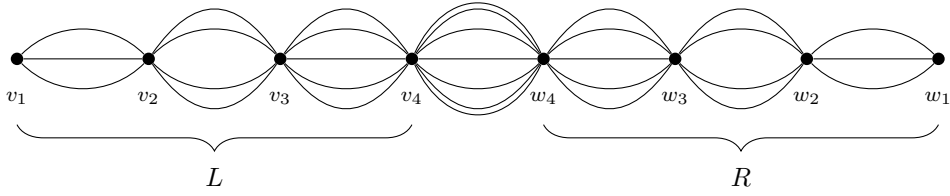
\begin{figure}
    \centering

\begin{tikzpicture}[scale=1.75]
\begin{scope}
\tikzstyle{every node}=[circle,fill,scale=0.5pt]
\node (0) at (0,0) {};
\node (1) at (1,0) {};
\node (2) at (2,0) {};
\node (3) at (3,0) {};
\node (4) at (4,0) {};
\node (5) at (5,0) {};
\node (6) at (6,0) {};
\node (7) at (7,0) {};
\end{scope}

\draw (0,-0.3) node {\footnotesize$v_1$};
\draw (1,-0.3) node {\footnotesize$v_2$};
\draw (2,-0.3) node {\footnotesize$v_3$};
\draw (3,-0.3) node {\footnotesize$v_4$};
\draw (4,-0.3) node {\footnotesize$w_4$};
\draw (5,-0.3) node {\footnotesize$w_3$};
\draw (6,-0.3) node {\footnotesize$w_2$};
\draw (7,-0.3) node {\footnotesize$w_1$};
\draw [decorate, decoration = {brace,amplitude=10pt}] (3,-0.5) --  (0,-0.5);
\draw [decorate, decoration = {brace,amplitude=10pt}] (7,-0.5) --  (4,-0.5);
\draw (1.5,-0.9) node {$L$};
\draw (5.5,-0.9) node {$R$};

\draw (0) -- (1);
\draw (0) to [out=45,in=135,looseness=1] (1);
\draw (0) to [out=315,in=225,looseness=1] (1);

\draw (1) to [out=45,in=135,looseness=1] (2);
\draw (1) to [out=315,in=225,looseness=1] (2);
\draw (1) to [out=55,in=125,looseness=1.5] (2);
\draw (1) to [out=305,in=235,looseness=1.5] (2);

\draw (2) -- (3);
\draw (2) to
[out=45,in=135,looseness=1] (3);
\draw (2) to [out=315,in=225,looseness=1] (3);
\draw (2) to [out=55,in=125,looseness=1.5] (3);
\draw (2) to [out=305,in=235,looseness=1.5] (3);

\draw (3) -- (4);
\draw (3) to
[out=45,in=135,looseness=1] (4);
\draw (3) to [out=315,in=225,looseness=1] (4);
\draw (3) to [out=55,in=125,looseness=1.5] (4);
\draw (3) to [out=305,in=235,looseness=1.5] (4);
\draw (3) to [out=65,in=115,looseness=1.5] (4);
\draw (3) to [out=295,in=245,looseness=1.5] (4);

\draw (4) -- (5);
\draw (4) to [out=45,in=135,looseness=1] (5);
\draw (4) to [out=315,in=225,looseness=1] (5);
\draw (4) to [out=55,in=125,looseness=1.5] (5);
\draw (4) to [out=305,in=235,looseness=1.5] (5);

\draw (5) to [out=45,in=135,looseness=1] (6);
\draw (5) to [out=315,in=225,looseness=1] (6);
\draw (5) to [out=55,in=125,looseness=1.5] (6);
\draw (5) to [out=305,in=235,looseness=1.5] (6);

\draw (6) -- (7);
\draw (6) to [out=45,in=135,looseness=1] (7);
\draw (6) to [out=315,in=225,looseness=1] (7);

\end{tikzpicture}

\caption{The symmetric generalized banana graph $B^{0,0}_{4,5,7}$.}
    \label{fig:SymBanana}
\end{figure}

As we will see in Lemma~\ref{Lem:Sym1} and Corollary~\ref{Cor:Sym2} below, the 1st and 2nd gonality of $B^{0,0}_{a,b,k}$ are both even.  To obtain more gonality sequences, we also consider generalized banana graphs that are ``almost'' symmetric.  Let $B^{0,1}_{a,b,k}$ be the graph obtained by connecting the vertex $v_a$ in $L = B^*_{a,b-1}$ to the vertex $w_a$ in $R = B^*_{a,b}$ by $k$ edges.  Let $B^{1,0}_{a,b,k}$ be the graph obtained by connecting the vertex $v_{a-1}$ in $L = B^*_{a-1,b}$ to the vertex $w_a$ in $R = B^*_{a,b}$ by $k$ edges.  Finally, let $B^{1,1}_{a,b,k}$ be the graph obtained by connecting the vertex $v_{a-1}$ in $L = B^*_{a-1,b-1}$ to the vertex $w_a$ in $R = B^*_{a,b}$ by $k$ edges.

We begin by computing the first gonalities of these graphs.

\begin{lemma}
\label{Lem:Sym1}
If $2 \leq a \leq b \leq 2a-1$ and $k \geq 2a$, then
\begin{align*}
\gon_1 (B^{0,0}_{a,b,k}) = \gon_1 (B^{0,1}_{a,b,k}) &= 2a \\
\gon_1 (B^{1,0}_{a,b,k}) = \gon_1 (B^{1,1}_{a,b,k}) &= 2a-1 .
\end{align*}
\end{lemma}

\begin{proof}
This follows directly from Lemma~\ref{Lem:BStar} and Proposition~\ref{Prop:ClosedUnderSum}.
\end{proof}

To compute the 2nd gonality of these graphs, we will need the following refinement of Proposition~\ref{Prop:ClosedUnderSum}.

\begin{proposition}
\label{Prop:RefClosedUnderSum}
Let $G_1$ and $G_2$ be graphs, let $v_1 \in V(G_1)$, $v_2 \in V(G_2)$, and let $G$ be the graph obtained by connecting $v_1$ to $v_2$ with $\ell$ edges.  If
\[
\ell \geq \gon_2 (G_1) + \gon_2 (G_2) - \min \{ \gon_1 (G_1) , \gon_1 (G_2) \} ,
\]
then
\[
\gon_2 (G) = \gon_2 (G_1) + \gon_2 (G_2) .
\]
\end{proposition}

\begin{proof}
Let $D_1$ be a divisor of rank 2 and degree $\gon_2 (G_1)$ on $G_1$, and let $D_2$ be a divisor of rank 2 and degree $\gon_2 (G_2)$ on $G_2$.  Then the divisor $D_1 + D_2$ has rank at least 2 on $G$, so $\gon_2 (G) \leq \gon_2 (G_1) + \gon_2 (G_2)$.

For the reverse inequality, let $D$ be a divisor of rank at least 2 on $G$.  We must show that $\Deg (D) \geq \gon_2 (G_1) + \gon_2 (G_2)$.  Without loss of generality, assume that $D$ is $v_1$-reduced.  Since $D$ has rank at least 2, $\Deg (D \vert_{G_1}) \geq \gon_2 (G_1)$ by Lemma~\ref{Lem:GlueReduced}.  We proceed by cases.  First, if $\Deg (D \vert_{G_2}) \geq \gon_2 (G_2),$ then $\Deg(D) \geq \gon_2 (G_1) + \gon_2 (G_2)$.
 
Second, if $\gon_1 (G_2) \leq \Deg (D \vert_{G_2}) < \gon_2 (G_2)$, then $(D \vert_{G_2})$ does not have rank at least 2, so we must be able to pass chip across the $\ell$ edges.  Thus, $\Deg(D) \geq \ell + \gon_1 (G_2) \geq \gon_2 (G_1) + \gon_2 (G_2)$.
 
Finally, if $\Deg (D \vert_{G_2}) < \gon_1 (G_2)$, then $D \vert_{G_2}$ does not have positive rank.  Again, we must be able to pass chips across the $\ell$ edges.  Let $D'$ be the divisor obtained by firing the subset of vertices $V(G_1)$.  If $E$ is the sum of a vertex of $G_1$ and a vertex of $G_2$, then $D-E$ is equivalent to an effective divisor.  If follows that $D' \vert_{G_1}$ must have positive rank.  Thus, $\Deg (D' \vert_{G_1}) \geq \gon_1 (G_1)$, so $\Deg (D') \geq \ell + \gon_1 (G_1) \geq \gon_2 (G_1) + \gon_2 (G_2)$.
\end{proof}

\begin{corollary}
\label{Cor:Sym2}
If $2 \leq a \leq b \leq 2a-1$ and $k \geq 2b-a+3$, then
\begin{align*}
\gon_2 (B^{0,0}_{a,b,k}) = \gon_2 (B^{1,0}_{a,b,k}) &= 2b+2 \\
\gon_2 (B^{0,1}_{a,b,k}) = \gon_2 (B^{1,1}_{a,b,k}) &= 2b+1 . 
\end{align*}
\end{corollary}

\begin{proof}
This follows directly from Lemma~\ref{Lem:BStar} and Proposition~\ref{Prop:RefClosedUnderSum}.
\end{proof}

We now compute the 3rd gonalities of these graphs.

\begin{theorem}
\label{Thm:SymBanana3Gon}
Let $2 \leq a \leq b \leq 2a-1$ and let $2b \leq k$.  We have the following:
\begin{enumerate}
\item if $k \leq 2a + b-1$, then $\gon_3 (B^{0,0}_{a,b,k}) = k+b+1$,
\item if $k \leq 2a + b-1$, then $\gon_3 (B^{0,1}_{a,b,k}) = k + b$,
\item if $k \leq 2a + b -2$, then $\gon_3 (B^{1,0}_{a,b,k}) = k + b + 1$, and
\item if $k \leq 2a + b - 2$, then $\gon_3 (B^{1,1}_{a,b,k}) = k + b$.
\end{enumerate}
\end{theorem}

\begin{proof}
We prove this in the case of $B^{0,0}_{a,b,k}$.  The other graphs are similar.  First note that the divisor $k \cdot (v_a) + (b+1) \cdot (w_a)$ has rank at least 3.  This shows that $\gon_3 (B^{0,0}_{a,b,k}) \leq b + k + 1$.  For the reverse inequality, let $D$ be a divisor of rank at least $3$ on $B^{0,0}_{a,b,k}$, and assume that $D$ is $v_a$-reduced.  By Lemma~\ref{Lem:GlueReduced} and Theorem~\ref{Thm:Banana3gon}, we have $\Deg({D \vert}_L) \geq a + b$.  We proceed by cases.

First, if $\Deg(D \vert_R) \geq a + b$, then $\Deg(D) \geq 2a + 2b \geq k + b + 1$.

Next, if $b+1 \leq \Deg(D \vert_R) < a +b$, then $D \vert_R$ has rank at most 2, so we must be able to pass chips across the $k$ edges.  It follows that $\Deg (D \vert_L) \geq k$, so $\Deg (D) \geq k+b+1$.

Third, if $a \leq \Deg(D \vert_R) < b +1,$ then $D \vert_R$ has rank at most 1, so we must able to pass chips across the $k$ edges.  Let $D'$ be the divisor obtained by firing the subset of vertices $V(L)$.  If $E$ is an effective divisor with $\Deg (E \vert_L) = 1$ and $\Deg (E \vert_R) = 2$, then $D-E$ is equivalent to an effective divisor.  It follows that $D' \vert_L$ must have rank at least 1.  Thus, $\Deg(D) \geq 2a + k \geq k + b + 1$.

Finally, if $\Deg(D \vert_R) < a$, then $D \vert_R$ does not have positive rank, so we must be able to pass chips across the $k$ edges.  Again, let $D'$ be the divisor obtained by firing the subset of vertices $V(L)$.  If $E$ is an effective divisor with $\Deg (E \vert_L) = 2$ and $\Deg (E \vert_R) = 1$, then $D-E$ is equivalent to an effective divisor.  It follows that $D' \vert_L$ must have rank at least 2.  Thus, $\Deg(D) \geq k + b + 1$.
\end{proof}

We can use these graphs to identify a large collection of sequences in $\mathcal{G}_3$.

\begin{theorem}
\label{Thm:SymBanana}
If $(x,y,z) \in \mathcal{C}_3$ satisfies $y \geq x+2$, $\frac{3}{2}y \leq z+2 \leq x+y$, then $(x,y,z) \in \mathcal{G}_3$.
\end{theorem}

\begin{proof}
If $x$ and $y$ are both even, consider the graph $B^{0,0}_{a,b,k}$ with $a = \frac{1}{2}x$, $b=\frac{1}{2}(y-2)$, and $k= z - \frac{1}{2}y$.  Since $y \geq x+2$, we have $a \leq b$.  Since $z \leq x+y-2$, we have $k \leq 2a+b-1$, and since $z \geq \frac{3}{2}y-2$, we have $2b \leq k$.  By Theorem~\ref{Thm:SymBanana3Gon}, the first 3 terms of the gonality sequence of $B^{0,0}_{a,b,k}$ are
$(2a,2b+2,k+b+1)=(x,y,z)$.

Similarly, if $x$ is even and $y$ is odd, consider the graph $B^{0,1}_{a,b,k}$ with $a = \frac{1}{2}x$, $b=\frac{1}{2}(y-1)$, and $k=z-\frac{1}{2}(y-1)$.  If $x$ is odd and $y$ is even, consider the graph $B^{1,0}_{a,b,k}$ with $a=\frac{1}{2}(x+1)$, $b=\frac{1}{2}(y-2)$, and $k=z-\frac{1}{2}y$.  Finally, if $x$ and $y$ are both odd, consider the graph $B^{1,1}_{a,b,k}$ with $a=\frac{1}{2}(x+1)$, $b=\frac{1}{2}(y-1)$, and $k=z-\frac{1}{2}(y-1)$.
\end{proof}

\begin{corollary}
\label{Cor:SymBPlusRook}
If $2a+2 \leq b \leq 3a-1$, $b \neq 2a+3$, then $(2a,b,3a+1) \in \mathcal{G}_3$.
\end{corollary}

\begin{proof}
If $b = 2a+2$, then $(2a,2a+2,3a+1) \in \mathcal{G}_3$ by Theorem~\ref{Thm:SymBanana}.  If $2a+4 \leq b$ and $m=b-2a-1$, then $m \geq 3$.  By Lemma~\ref{Lem:Rook}, the first 3 terms of the gonality sequence of the rook graph $K_3 \square K_m$ are $(2m, 3m-1, 3m) \in \mathcal{G}_3$.  If $n = 3a-b+1$, then $n \geq 2$.  By Theorem~\ref{Thm:SymBanana}, we have $(2n, 2n+2, 3n+1) \in \mathcal{G}_3$.  Thus, by Theorem~\ref{Thm:Closed}, we have
\begin{align*}
(2m, 3m-1, 3m) + (2n, 2n+2, 3n+1) &= (2(n+m), 3m+2n+1, 3(m+n)+1) \\
&= (2a, b, 3a+1) \in \mathcal{G}_3 .
\end{align*}
\end{proof}

We now prove Theorem~\ref{Thm:G3zLess}.

\begin{proof}[Proof of Theorem~\ref{Thm:G3zLess}]
If $z \geq 2x$, then this follows from Theorem~\ref{Thm:G3zBig}.  For the remainder of the proof, we therefore assume that $z < 2x$.

Next, consider the case where $y=x+2$.  By Theorem~\ref{Thm:SymBanana}, if $\frac{3}{2}x +1 \leq z \leq 2x$, then $(x,x+2,z) \in \mathcal{G}_3$.  For the remainder of the proof, we assume that $y \geq x+3$.

Next, consider the cases where $z \geq 2x-2$.  If $z = 2x-1$, then since $\frac{3}{2}x+2 \leq z$, we have $x \geq 6$, and if $z = 2x-2$, then then since $\frac{3}{2}x+2 \leq z$, we have $x \geq 8$.  For $x \leq 7$, the possibilities are: $(x,y,z) = (6,8,11), (6,9,11), (7,9,13), (7,10,13), (7,11,13)$.  All of these except for $(7,11,13)$ are in $\mathcal{G}_3$ by Theorem~\ref{Thm:SymBanana}.  To see that $(7,11,13) \in \mathcal{G}_3$, note that $(3,5,6) \in \mathcal{G}_3$ by Theorem~\ref{Thm:G3zBig}, and $(4,6,7) \in \mathcal{G}_3$ by the third graph in the right column of \cite[Table~4.1]{ADMYY}.  By Theorem~\ref{Thm:Closed}, $(3,5,6) + (4,6,7) = (7,11,13) \in \mathcal{G}_3$.  For $8 \leq x \leq y-3$, by Theorem~\ref{Thm:G3zBig}, we have $(x-6, y-8, 2x-11), (x-6, y-8, 2x-10) \in \mathcal{G}_3$, and by Lemma~\ref{Lem:Rook}, we have $(6,8,9) \in \mathcal{G}_3$.  Thus, by Theorem~\ref{Thm:Closed}, $(x,y, 2x-2), (x,y,2x-1) \in \mathcal{G}_3$ as well.  For the remainder of the proof, we assume that $z < 2x-2$.

We now consider the cases where $3x \leq y+z$.  Let $a = 2z-3x$, $b=2x-z$, and $c=y+3z-6x+1$.  Since $z \geq \frac{3}{2}x+2$, we have $a \geq 2$.  Since $y \leq z-2$, we have $c \leq 2a-1$, and since $3x \leq y+z$, we have $c \geq a+1$.  It follows from Theorem~\ref{Thm:G3zBig} that $(a, c, 2a) \in \mathcal{G}_3$.  Similarly, since $z \leq 2x-3$, we have $b \geq 3$.   By Lemma~\ref{Lem:Rook}, the first 3 terms of the gonality sequence of the rook graph $K_3 \square K_b$ are $(2b, 3b-1, 3b) \in \mathcal{G}_3$.  Thus, by Theorem~\ref{Thm:Closed}, we have
\begin{align*}
(a, c, 2a) + (2b, 3b-1, 3b) &= (a+2b, 3b+c-1, 2a+3b) \\
&= (x, y, z) \in \mathcal{G}_3 .
\end{align*}

We now consider the remaining cases.  Since $y \geq x+3$ and $3x \geq y+z+1$, we see that $z \leq 2x-4$.  Similarly, since $z \geq \frac{3}{2}x + 2$, we have $y \leq 3x-z-1 \leq \frac{3}{2}x - 3 \leq z-5$.  If $a = 2z-3x-2$, then $a \geq 2$, so by Theorem~\ref{Thm:G3zBig}, we have $(a,c,2a) \in \mathcal{G}_3$ for all $c$ in the range $a+1 \leq c \leq 2a-1$.  If $b = 2x-z+1$, then since $z \leq 2x-4$, we have $b \geq 5$.  Thus, by Corollary~\ref{Cor:SymBPlusRook}, $(2b,d,3b+1) \in \mathcal{G}_3$ for all $d$ in the range $2b+2 \leq d \leq 3b-1$, $d \neq 2b+3$.  If $a>2$, we can choose $c$ and $d$ so that $c+d$ can take any integer value in the range
\[
x+3 = (a+1) + (2b+2) \leq c+d \leq (2a-1) + (3b-1) = z-3.
\]
If $a=2$, then $c$ must be 3, and we cannot choose $d$ so that $c+d = 2b+3$.  However, in this case we have $y=x+4$, and the sequence $(x,x+4,z)$ is in $\mathcal{G}_3$ by Theorem~\ref{Thm:G3zBig}.  Otherwise, since $x+3 \leq y \leq z-5$, we may choose $c$ and $d$ so that $c+d = y$.  Thus, by Theorem~\ref{Thm:Closed}, we have
\begin{align*}
(a, c, 2a) + (2b, d, 3b+1) &= (a+2b, c+d, 2a+3b+1) \\
&= (x, y, z) \in \mathcal{G}_3 .
\end{align*}
\end{proof}

\section{Gonality Sequences of Algebraic Curves}
\label{Sec:Curves}

By Theorem~\ref{Thm:G2}, the semigroup $\mathcal{G}_r$ is not finitely generated for any $r \geq 2$.  Indeed, if $\vec{x} \in \mathcal{G}_r$ and $x_{i+1} = x_i + 1$ for some $i$, then $\vec{x}$ is irreducible.  As we have seen in Theorem~\ref{Thm:Closed}, if $\vec{x} \in \mathcal{G}_r$ is reducible, then there exists graphs of arbitrarily large genus with gonality sequence $\vec{x}$.  Irreducible elements of $\mathcal{G}_r$ are more mysterious.  In this final section, we study the gonality sequences of algebraic curves $C$ such that $\gon_r (C) = \gon_{r-1} (C) + 1$ for some $r$.  These curves have interesting properties, and we ask whether graphs with the same gonality sequence exhibt the same properties.

\begin{lemma}
\label{Lem:EmbeddedCurve}
Let $C$ be a smooth curve and let $r$ be a positive integer.  If $\gon_r (C) = \gon_{r-1} (C) + 1$, then $C$ is isomorphic to a smooth curve of degree $\gon_r (C)$ in $\PP^r$.
\end{lemma}

\begin{proof}
Let $\cL$ be a line bundle on $C$ of rank $r$ and degree $\gon_r (C)$.  Let $\varphi_{\cL} \colon C \to \PP^r$ be the map given by the complete linear series of $\cL$, let $B = \varphi_{\cL} (C)$ be the image, let $\nu \colon \widetilde{B} \to B$ be the normalization of $B$, and let $\varphi \colon C \to \widetilde{B}$ be the induced map.

We first show that the map $\varphi$ has degree 1, and is therefore an isomorphism.  For any point $p \in \widetilde{B}$, the line bundle $\nu^* \cO_B (1)(-p)$ has rank at least $r-1$ on $\widetilde{B}$.  Thus, $\varphi^* \nu^* \cO_B (1)(-p)$ has rank at least $r-1$ on $C$.  But
\[
\deg (\varphi^* \nu^* \cO_B (1)(-p)) = \deg (\cL) - \deg (\varphi) = \gon_r (C) - \deg (\varphi) .
\]
Since $\gon_{r-1} (C) = \gon_r (C) - 1$, it follows that $\deg (\varphi) = 1$.

We now show that the map $\nu$ is an isomorphism.  If not, then $B$ is singular, and projection from a singular point yields a nondegenerate map to $\PP^{r-1}$ of degree at most $\gon_r (C) - 2$.  Since $\gon_{r-1} (C) = \gon_r (C) - 1$, this is again impossible.  It follows that the map $\varphi_{\cL}$ is an isomoprhism onto its image.
\end{proof}

Lemma~\ref{Lem:EmbeddedCurve} has several consequences.

\begin{lemma}
\label{Lem:PlaneCurves}
Let $C$ be a curve with the property that $\gon_2 (C) = \gon_1 (C) + 1$.  Then the genus of $C$ is $g = {{\gon_1 (C)}\choose{2}}$ and, for any $r<g$, we have
\[
\gon_r (C) = k \cdot \gon_2 (C) - h,
\]
where $k$ and $h$ are the uniquely determined integers with $1 \leq k \leq \gon_2(C) - 3$, $0 \leq h \leq k$, such that $r = \frac{k(k+3)}{2} - h$.

In particular, if $\gon_1 (C) \geq 2$, then $\gon_3 (C) = 2 \cdot \gon_1 (C)$.
\end{lemma}

\begin{proof}
By Lemma~\ref{Lem:EmbeddedCurve}, $C$ is isomorphic to a smooth plane curve of degree $\gon_2 (C)$.  The genus of such a curve is ${{\gon_1 (C)}\choose{2}}$, and its gonality sequence is computed in \cite{Noether82, Hartshorne86}.
\end{proof}

\begin{lemma}
\label{Lem:SpaceCurves}
Let $C$ be a curve with the property that $\gon_3 (C) = \gon_2 (C) + 1$, and let $m = \lceil \frac{1}{2} \gon_2 (C) \rceil$.  Then the genus of $C$ is at most $m \cdot \gon_3 (C) - m(m+2)$.  Moreover, if equality holds, then 
\[
\gon_1 (C) = \Big\lceil \frac{1}{2} (\gon_3 (C) - 1) \Big\rceil.
\]
\end{lemma}

\begin{proof}
By Lemma~\ref{Lem:EmbeddedCurve}, $C$ is isomorphic to a smooth space curve of degree $\gon_3 (C)$.  By \cite[Theorem~IV.6.7]{Hartshorne77}, the genus of $C$ is at most $m \cdot \gon_3 (C) - m(m+2)$, and if equality holds, then $C$ is contained in a quadric surface.  A tangent plane to the quadric meets it in two lines, which meet the curve $C$ in $\gon_3 (C)$ points.  It follows that one of these two lines must meet $C$ in at least $\frac{1}{2} \gon_3 (C)$ points, and projection from this line yields a nondegenerate map to $\PP^1$ of degree at most $\frac{1}{2} \gon_3 (C)$. Thus,
\[
\gon_1 (C) \leq \frac{1}{2} \gon_3 (C) .
\]
On the other hand, we have
\[
\gon_1 (C) \geq \frac{1}{2} \gon_2 (C) = \frac{1}{2} (\gon_3 (C) -1),
\]
and the result follows.
\end{proof}

\begin{question}
\label{Q:SpaceGraphs}
Let $G$ be a graph with the property that $\gon_3 (G) = \gon_2 (G) + 1$, and let $m = \lceil \frac{1}{2} \gon_2 (G) \rceil$.
\begin{enumerate}
\item Must the genus of $G$ be at most $m \cdot \gon_3 (G) - m(m+2)$?
\item If equality holds, is it true that 
\[
\gon_1 (C) = \Big\lceil \frac{1}{2} (\gon_3 (C) - 1) \Big\rceil ?
\]
\end{enumerate}
\end{question}

\begin{lemma}
\label{Lem:SpaceCurveGonality}
Let $C$ be a curve.  If $\gon_3 (C) \leq \gon_1 (C) + 3$, then $\gon_1 (C) \leq 6$ and $\gon_1 (C) \neq 5$.  
\end{lemma}

\begin{proof}
Suppose that $\gon_3 (C) \leq \gon_1 (C) + 3$.  Then either $\gon_2 (C) = \gon_1 (C) + 1$ or $\gon_3 (C) = \gon_2 (C) + 1$.  If $\gon_2 (C) = \gon_1 (C) + 1$, then by Lemma~\ref{Lem:PlaneCurves}, 
\[
2 \gon_1 (C) = \gon_3 (C) \leq \gon_1 (C) + 3,
\]
hence $\gon_1 (C) \leq 3$.

If $\gon_3 (C) = \gon_2 (C) + 1$, then by Lemma~\ref{Lem:EmbeddedCurve}, $C$ is isomorphic to a smooth space curve of degree $\gon_3 (C)$.  By \cite[Proposition~4.1]{HartshorneSchlesinger11}, if $\gon_3 (C) \geq 10$, then $\gon_3 (C) \geq \gon_1 (C) + 4$, hence we must have $\gon_3 (C) \leq 9$.

It remains to show that, if $\gon_3 (C) = 8$, then $\gon_1 (C) \leq 4$.  Since every curve of genus 6 or less has gonality at most 4, we may assume that $C$ has genus at least 7.  By Lemma~\ref{Lem:SpaceCurves}, if $\gon_3 (C) = 8$, then $C$ has genus at most 9, and if it is equal to 9, then $\gon_1 (C) \leq 4$.  If $C$ has genus 8, then $\cO_C(2)$ has degree $16 > 2 \cdot 8 - 2$, hence $h^0 (C, \cO_C (2)) = 9$.  It follows that $C$ is contained in a quadric surface, and again, $\gon_1 (C) \leq \frac{1}{2}\gon_3 (C) = 4$.  Finally, if $C$ has genus 7, then by Riemann-Roch, $K_C \otimes \cO_C (-1)$ has degree 4 and rank 1, hence $\gon_1 (C) \leq 4$.
\end{proof}

\begin{question}
\label{Q:SpaceCurveGonality}
Let $G$ be a graph.  If $\gon_1 (G) = 5$ or $\gon_1 (G) \geq 7$, does it follow that $\gon_3 (G) \geq \gon_1 (G) + 4$?
\end{question}

\bibliography{math}

\newcommand{\etalchar}[1]{$^{#1}$}
\begin{thebibliography}{ADM{\etalchar{+}}21}

\bibitem[ADM{\etalchar{+}}21]{ADMYY}
Ivan Aidun, Frances Dean, Ralph Morrison, Teresa Yu, and Julie Yuan.
\newblock Gonality sequences of graphs.
\newblock {\em SIAM J. Discrete Math.}, 35(2):814--839, 2021.

\bibitem[Bak08]{Baker08}
M.~Baker.
\newblock Specialization of linear systems from curves to graphs.
\newblock {\em Algebra Number Theory}, 2(6):613--653, 2008.

\bibitem[BN09]{BakerNorine09}
M.~Baker and S.~Norine.
\newblock Harmonic morphisms and hyperelliptic graphs.
\newblock {\em Int. Math. Res. Not.}, 2009(15):2914--2955, 2009.

\bibitem[CDJP19]{CDJP}
Filip Cools, Michele D'Adderio, David Jensen, and Marta Panizzut.
\newblock Brill-{N}oether theory of curves on {$\Bbb P^1\times\Bbb P^1$}:
  tropical and classical approaches.
\newblock {\em Algebr. Comb.}, 2(3):323--341, 2019.

\bibitem[CP17]{CoolsPanizzut}
Filip Cools and Marta Panizzut.
\newblock The gonality sequence of complete graphs.
\newblock {\em Electron. J. Combin.}, 24(4):Paper No. 4.1, 20, 2017.

\bibitem[Har77]{Hartshorne77}
R.~Hartshorne.
\newblock {\em Algebraic geometry}.
\newblock Springer-Verlag, New York, 1977.
\newblock Graduate Texts in Mathematics, No. 52.

\bibitem[Har86]{Hartshorne86}
Robin Hartshorne.
\newblock Generalized divisors on {G}orenstein curves and a theorem of
  {N}oether.
\newblock {\em J. Math. Kyoto Univ.}, 26(3):375--386, 1986.

\bibitem[HS11]{HartshorneSchlesinger11}
Robin Hartshorne and Enrico Schlesinger.
\newblock Gonality of a general {ACM} curve in {$\Bbb P^3$}.
\newblock {\em Pacific J. Math.}, 251(2):269--313, 2011.

\bibitem[Noe82]{Noether82}
M.~Noether.
\newblock Zur {G}rundlegung der {T}heorie der algebraischen {R}aumcurven.
\newblock {\em J. Reine Angew. Math.}, 93:271--318, 1882.

\bibitem[Spe22]{Speeter}
N.~Speeter.
\newblock Ther gonality of rook graphs.
\newblock preprint arXiv:2203.16619, 2022.

\end{thebibliography}

\end{document}